\documentclass{article} % replace ECP by EJP if needed.
\title{A central limit theorem for continuous-time Markov processes conditioned not to be absorbed}
\author{William Oçafrain\footnote{Université de Lorraine, CNRS, Inria, IECL, F-54000, Nancy, France. Email : w.ocafrain@hotmail.fr}}%AUTHORS
\usepackage[T1]{fontenc}
\usepackage[utf8]{inputenc}
\usepackage{lmodern}
\usepackage{amsmath,amssymb}
\usepackage[a4paper]{geometry}
\usepackage{amsthm}
\usepackage[british]{babel}
\usepackage{graphicx}
\usepackage{xcolor}
\usepackage{dsfont}
\usepackage{enumerate}
\usepackage{enumerate}
\usepackage{bbm}
\usepackage{bigints}
\usepackage{color}

\usepackage{tikz}
\usepackage{comment}

\def\Cov{\textbf{Cov}}

\def\d{\partial}

\newtheorem{assumption}{Assumption}
\newtheorem{theorem}{Theorem}
\newtheorem{lemma}{Lemma}

\newtheorem{remark}{Remark}
\newtheorem{corollary}{Corollary}

\def\Kolm{\text{Kolm}}
\def\L{\mathbb{L}}
\def\E{\mathbb{E}}
\def\P{\mathbb{P}}
\def\R{\mathbb{R}}
\def\Q{\mathbb{Q}}
\def\N{\mathbb{N}}
\def\1{\mathbbm{1}}
\def\d{\partial}
\def\Z{\mathbb{Z}}

\def\cE{{\mathcal E}}
\def\cF{{\mathcal F}}

\def\cD{\mathcal{D}}

\def\cM{\mathcal{M}}
\def\cB{\mathcal{B}}

\def\cN{\mathcal{N}}
\def\Kolm{\text{Kolm}}
\begin{document}

\maketitle

\abstract{This paper aims to establish a central limit theorem for Markov processes conditioned not to be absorbed under a very general assumption on quasi-stationarity for the underlying process. To do so, a central limit theorem has been established for ergodic Markov processes. The conditional central limit theorem is then obtained by applying the central limit theorem to the $Q$-process.     }

\textit{Key words: Quasi-stationary distribution; Quasi-stationarity; Quasi-ergodic distribution; Central limit theorem; $Q$-process.}

%%%%%%%%%%%%%%%%%%%%%%%%%%%%%%%%%%%%%%%%%%%%%%%%%%%%%%%%%%%%%%%%%%%
%%                                                               %%
%% No need for \maketitle.                                       %%
%%                                                               %%
%%%%%%%%%%%%%%%%%%%%%%%%%%%%%%%%%%%%%%%%%%%%%%%%%%%%%%%%%%%%%%%%%%%

%%%%%%%%%%%%%%%%%%%%%%%%%%%%%%%%%%%%%%%%%%%%%%%%%%%%%%%%%%%%%%%%%%%
%%                                                               %%
%% Please replace what follows by the body of your article       %%
%% (up to the bibliography):                                     %%
%%                                                               %%
%%%%%%%%%%%%%%%%%%%%%%%%%%%%%%%%%%%%%%%%%%%%%%%%%%%%%%%%%%%%%%%%%%%

\section*{Notation}

\begin{itemize}
    \item $\cM_1(E)$: Set of the probability measures defined on $E$. 
    \item For any $\mu \in \cM_1(E)$ and measurable function $f$ such that $\int_E f(x) \mu(dx)$ is well-defined,
    $$\mu(f) := \int_E f(x) \mu(dx).$$
\item For a given positive function $\psi$, $\L^\infty(\psi)$ is the set of functions $f$ such that $f/\psi$ is bounded, endowed with the norm 
$$\|f\|_{\L^\infty(\psi)} := \|f/\psi\|_\infty.$$ 
    \item For any positive measurable function $\psi$, for any $\mu, \nu \in \cM_1(E)$,
    $$\|\mu - \nu \|_{\psi} := \sup_{\|f\|_{\L^\infty(\psi)} \leq 1} |\mu(f) - \nu(f)|.$$
    \item For any nonnegative measurable function $f$ and $\mu \in \cM_1(E)$ such that $\mu(f) \in (0,+\infty)$,
    $$f \circ \mu(dx) := \frac{f(x) \mu(dx)}{\mu(f)}.$$
\item Kolmogorov distance: For any $\mu, \nu \in \cM_1(\R)$,
$$d_\Kolm(\mu, \nu) := \sup_{x \in \R} |\mu((-\infty,x]) - \nu((-\infty,x])|.$$
\end{itemize}

\section{Introduction}

\subsection{Introduction to quasi-stationarity}

Let $(X_t)_{t \geq 0}$ be a time-homogeneous continuous-time Markov process living on a state space $(E \cup \{\d\}, \cE)$, where $\d \not \in E$ is an absorbing state for the process $X$, which means that $X_{t} = \d$ conditioned to $\{X_s = \d\}$ for all $s \leq t$, and $\cE$ is a $\sigma$-field associated to the state space $E$\footnote{$X$ is assumed to satisfy the time-homogeneous Markov property with respect to its natural filtration. $E$ and sample paths of $X$ are assumed to be equipped with appropriate $\sigma$-fields; for instance, and most usually, there may exist a Polish topology on $E$ under which $X$ is cadlag, although topology plays no role in the present work.}. Denote by $\tau_\d$ the hitting time of $\d$ by the process $X$. We associate to the process $X$ a family of probability measure $(\P_x)_{x \in E \cup \{\d\}}$ such that $\P_x[X_0 = x] = 1$ for any $x \in E \cup \{\d\}$. For any probability measure $\mu \in \cM_1(E \cup \{\d\})$, define $\P_\mu := \int_{E \cup \{\d\}} \mu(dx) \P_x$, and denote $\E_x$ and $\E_\mu$ the associated expectations. Moreover, denote by $(\cF_t)_{t \geq 0}$ the natural filtration of the process $X$.

In this paper, we assume that the process $X$ admits a \textit{quasi-stationary distribution}, defined as a probability measure $\alpha \in \cM_1(E)$ such that, for all $t \geq 0$,
\begin{equation}\label{qsd}\P_\alpha[X_t \in \cdot | \tau_\d > t] = \alpha.\end{equation}
Such a probability measure is also a \textit{quasi-limiting distribution}, defined as a probability measure such that there exists a subset $\cD(\alpha) \subset \cM_1(E)$, called \textit{domain of attraction of $\alpha$}, such that, for all $\mu \in \cD(\alpha)$ and $A \in \cE$, 
$$\P_\mu[X_t \in A | \tau_\d > t] \underset{t \to \infty}{\longrightarrow} \alpha(A).$$
In particular, if $\alpha$ is a quasi-stationary distribution, $\alpha \in \cD(\alpha)$ by \eqref{qsd}. Conversely, we can show that any quasi-limiting distributions for $X$ satisfy \eqref{qsd} for all $t \geq 0$ (see \cite[Proposition 1]{MV2012}). In other terms, quasi-stationary and quasi-limiting distributions are equivalent notions. 

Denote by ${\lambda_0} := - \log(\P_\alpha[\tau_\d > 1])$. Then, it is well-known (see \cite[Proposition 2]{MV2012} for example) that, for all $t \geq 0$,
$$\P_\alpha[\tau_\d > t] = e^{-{\lambda_0} t},~~\forall t \geq 0.$$ 
A consequence of this property coupled with \eqref{qsd} is that, for all $t \geq 0$,
\begin{equation}\label{left-eigenmeasure}\P_\alpha[X_t \in \cdot, \tau_\d > t] = e^{-{\lambda_0} t} \alpha(\cdot).\end{equation}
Conversely, if a probability measure $\alpha$ satisfies \eqref{left-eigenmeasure} for a given $\lambda_0 > 0$, then $\alpha$ is a quasi-stationary distribution for the process $X$. In that respect, the quasi-stationary distributions for $X$ are exactly the probability left eigenmeasures for the semigroup $(P_t)_{t \geq 0}$ defined by
$$P_tf(x) := \E_x(f(X_t)\1_{\tau_\d > t}),$$
for all $t \geq 0$, $f$ belonging to a Banach space and $x \in E$. In what follows, we will use the notation
$$\mu P_t := \P_\mu(X_t \in \cdot, \tau_\d > t).$$

Also, we assume that the process $X$ admits a nonnegative function $\eta$ defined on $E$, vanishing at $\d$ and satisfying $\alpha(\eta) = 1$, such that, for all $x \in E$ and $t \geq 0$,
$$\E_x[\eta(X_t)\1_{\tau_\d > t}] = e^{-\lambda_0 t} \eta(x).$$ 
$\eta$ is therefore a right eigenfunction for the semigroup $(P_t)_{t \geq 0}$, associated to the eigenvalues $(e^{-\lambda_0 t})_{t \geq 0}$.

\subsection{The main assumption and the $Q$-process}

%More precisely, this quasi-stationary distribution has the following property: there exists $C > 0$ and $\rho \in (0,1)$ such that, for any probability measure $\mu \in \cM_1(E)$ such that $\mu(\psi_1) + \infty $ and $\mu(\psi_2) > 0$, 
%$$\| \P_\mu[X_n \in \cdot | \tau_\d > n] - \alpha \|_{TV} \leq C \rho^n,~~~~\forall n \in \Z_+,$$
%where $\| \cdot \|_{TV}$ refers to the total variation distance.

The main assumption on this process is the following.

\begin{assumption}
\label{assumption}
There exists a function $\psi_1 : E \to [1,+\infty)$, such that $\alpha(\psi_1) < + \infty$ and $\eta \in \L^\infty(\psi_1)$, as well as two constants $C, \gamma > 0$ such that, for any $\mu \in \cM_1(E)$ and $t \geq 0$,
    \begin{equation}\label{ergodicity}\|e^{{\lambda_0} t} \mu P_t - \mu(\eta) \alpha\|_{\psi_1} \leq C \mu(\psi_1) e^{-\gamma t}.\end{equation}
\end{assumption}
This assumption is satisfied under the general criteria Assumption (F) of \cite{CV2017c}. In particular, it is shown in \cite{CV2017c} that Assumption \ref{assumption} is satisfied for a lot of processes such as multidimensional elliptic diffusion processes or processes defined in discrete state space. In particular, we refer the reader to \cite[Sections 4 and 5]{CV2017c} for examples for which Assumption \ref{assumption} holds true. Assumption \ref{assumption} is also satisfied for general strongly Feller processes, as shown in \cite{guillin2020quasi}, and for some degenerate diffusion processes, as studied in \cite{benaim2021degenerate,lelievre2022quasi}. We refer the reader to \cite{CV2014,ferre2021more,velleret2018unique,bansaye2019non,occafrain2021convergence} for alternative criteria ensuring Assumption \ref{assumption}. \medskip

We can show (a short proof is provided later in the appendix of this paper) that Assumption \ref{assumption} implies the following one.
\begin{assumption}
\label{assumption-2}
\begin{enumerate}[i)]
\item 
Denoting $E' := \{x \in E : \eta(x) > 0\}$, the family of probability measures $(\Q_x)_{x \in E'}$ defined by
    $$\Q_x(\Gamma) := \lim_{T \to \infty} \P_x(\Gamma | \tau_\d > T),~~~~\forall t \geq 0, \forall \Gamma \in \cF_t,$$
    is well-defined. 
\item
     Under $(\Q_x)_{x \in E'}$, $X$ is a Markov process on $E'$ admitting $\beta(dx) := \eta(x) \alpha(dx)$ as an invariant probability measure. Moreover, denoting $$\psi(x) := \frac{\psi_1(x)}{\eta(x)},$$ $\beta(\psi) < + \infty$ and, for all $t \geq 0$ and $x \in E'$,
    \begin{equation}\label{ergodicity-q-process}\|\Q_x(X_t \in \cdot) - \beta\|_{\psi} \leq C \psi(x) e^{-\gamma t},\end{equation}
    where $C, \gamma > 0$ are the same constants as in \eqref{ergodicity}.
\end{enumerate} 
\end{assumption}

Since the process $X$ under $(\Q_x)_{x \in E'}$ is a Markov process, the family of operators $(Q_t)_{t \geq 0}$ defined by 
$$Q_tf(x) := \E_x^\Q(f(X_t)), ~~~~\forall t \geq 0, \forall x \in E', \forall f \in \L^\infty(\psi_1/\eta),$$
where $\E_x^\Q$ is the expectation associated to $\Q_x$, is a semigroup. In the literature  (see for example \cite[Theorem 2.7]{CV2017c}), the Markov process associated to this semigroup is called the \textit{$Q$-process}.  

Independently on the satisfaction of Assumption \ref{assumption}, \eqref{ergodicity-q-process} is satisfied when the $Q$-process satisfies the assumptions 1 and 2 in \cite{hairer2011yet}. Moreover, the inequality \eqref{ergodicity-q-process} implies, since $\eta \in \L^\infty(\psi_1)$, that, for all $x \in E'$ and $t \geq 0$,
\begin{equation*}\label{ergodicity-tv}\|\delta_x Q_t - \beta\|_{TV} \leq C \textcolor{black}{\|\eta\|_{\L^\infty(\psi_1)}}\frac{\psi_1(x)}{\eta(x)}e^{-\gamma t},\end{equation*}
where $\|\cdot\|_{TV}$ denotes the total variation norm. 

\subsection{The main result}

A consequence of Assumption \ref{assumption} is that the probability measure $\beta$ is a \textit{quasi-ergodic distribution} for the process $X$. That is, for all bounded measurable function $f$ and $\mu \in \cM_1(E)$ satisfying $\mu(\psi_1) < + \infty$ and $\mu(\eta) > 0$, the convergence
%Then, under Assumption (E), $\beta$ is actually a \textit{quasi-ergodic distribution} for the process $X$, that is to say that, for some initial distribution $\mu$ and any $f \in \cB(E)$,
\begin{equation}\label{qed}\E_\mu\left[\frac{1}{t} \int_0^t f(X_s)ds \middle| \tau_\d > t\right] \underset{t \to \infty}{\longrightarrow} \beta(f)
\end{equation}
holds true. This property is a consequence of the following lemma, whose the proof is postponed to the appendix of this paper.

\begin{lemma}
\label{lemma-q-proc}
For all $x \in E'$, $t \geq 0$ and $\Gamma \in \cF_t$,
\begin{equation} \label{expr-q-proc}\Q_x(\Gamma) = e^{\lambda_0 t} \frac{P_t [\eta \1_\Gamma](x)}{\eta(x)}.\end{equation}
Moreover, there exists a constant $C' > 0$ such that, for all $\mu \in \cM_1(E)$, $0 \leq t \leq T$ and $\Gamma \in \cF_t$, 
\begin{equation}\label{expo-conv}|\Q_{\eta \circ \mu}(\Gamma) - \P_\mu(\Gamma | \tau_\d > T)| \leq C' \frac{\mu(\psi_1)}{\mu(\eta)} e^{-\gamma (T-t)}.\end{equation}
\end{lemma}
In particular, the property \eqref{expo-conv} implies \eqref{qed} as shown in \cite{CV2016}. More precisely, we can show that the previous lemma implies the corollary below:
\begin{corollary}
For all $\mu \in \cM_1(E)$ such that $\mu(\psi_1) < + \infty$ and $\mu(\eta) > 0$, for all $f$ bounded,
$$\E_\mu\left(\left|\frac{1}{t}\int_0^t f(X_s)ds - \beta(f)\right|^2 \middle| \tau_\d > t\right) \underset{t \to \infty} {\longrightarrow} 0,$$
implying \eqref{qed} and that, for all $\mu \in \cM_1(E)$ and $f$ bounded, for all $\epsilon > 0$,
$$\P_\mu\left(\left|\frac{1}{t}\int_0^t f(X_s)ds - \beta(f)\right| \geq \epsilon \middle| \tau_\d > t\right) \underset{t \to \infty}{\longrightarrow} 0.$$
\end{corollary}
Provided Lemma \ref{lemma-q-proc}, a short proof can be obtained adapting the proofs in \cite{occafrain2020ergodic} or \cite{he2021exponential}, providing even $1/t$ as speed of convergence. \medskip

The aim of this paper is to prove a central limit theorem for processes satisfying Assumption \ref{assumption}, conditioned not to be absorbed up to the time $t$. Existing results stating a conditional central limit theorem for absorbing discrete-time Markov chains can be found in \cite{CMSM,matthews1970central,al1994central,szubarga1985functional,bolthausen1976functional,iglehart1974functional}. In particular, in \cite[Section 3.6]{CMSM}, it is stated that, for any Markov chain $(X_n)_{n \in \Z_+}$ defined on a finite state space $E \cup \{\d\}$ (absorbed at $\d$) whose matrix $(\P_i(X_1 = j))_{i,j \in E}$ is irreducible and aperiodic, one has that, for all  function $f$ such that $\beta(f) = 0$, the limit
$$\theta^2 := \lim_{n \to \infty} \frac{1}{n} \E_\alpha\left(\left(\sum_{k=0}^n f(X_k)\right)^2 \middle| \tau_\d > n\right)$$
is well-defined. If moreover $\theta^2 \ne 0$, one obtains
$$\lim_{n \to \infty} \P_\alpha \left[\frac{1}{\sqrt{n}} \sum_{k=0}^n f(X_k) \leq y \middle| \tau_\d > n \right] = 
    \int_{- \infty}^y \frac{1}{\sqrt{2\pi \theta^2}} e^{-\frac{x^2}{2 \theta^2}}dx,$$ for all $y \in \R$. This result is extended to all initial distributions in \cite{matthews1970central}, where it is also claimed that the limiting Gaussian distribution is the same as the one obtained in the central limit theorem applied to the $Q$-process (i.e. same limiting variance).   

The main result of this paper is then the following.
\begin{theorem}
\label{berry-esseen-1}
Assume that the process $(X_t)_{t \geq 0}$ satisfies Assumption \ref{assumption}. 

Then, for all $f \in \L^\infty(\1_{E})$ such that $\sigma_f^2 > 0$ and $\mu \in \cM_1(E)$ such that $\mu(\psi_1) < + \infty$ and $\mu(\eta) > 0$,
$$\P_\mu\left(\sqrt{t} \left[\frac{1}{t}\int_0^t f(X_s)ds - \beta(f)\right] \in \cdot \middle| \tau_\d > t\right) \underset{t \to + \infty}{\overset{w}{\longrightarrow}} \cN(0,\sigma_f^2),$$
where $w$ refers to the weak convergence of measures, where $\cN(0,\sigma_f^2)$ refers to the centered Gaussian variable of variance 
\begin{equation}\label{variance}\sigma_f^2 := 2\int_0^\infty \Cov_\beta^\Q(f(X_0),f(X_s))ds,\end{equation}
where $\Cov_\beta^\Q$ refers to the covariance with respect to the probability measure $\Q_\beta := \int_{E'} \beta(dx) \Q_x$. 
\end{theorem}

In particular, \eqref{ergodicity-q-process} implies that $\sigma_f^2 < + \infty$ for any $f$ bounded by $1$, since, assuming without loss of generality that $\beta(f) = 0$, for all $k \geq 0$,
\begin{equation}\label{majoration-1}|\E_\beta^\Q(f(X_0)f(X_k))| = |\E_\beta^\Q(f(X_0)\E_{X_0}^\Q(f(X_k))| \leq C \beta(\psi_1/\eta)e^{-\gamma k}.\end{equation}

This paper is only interested in processes conditioned not to be absorbed by absorbing states. Nevertheless, the following proofs can be adapted to general non-conservative semigroups satisfying Assumption \ref{assumption}. Some examples of such semigroups have been studied in \cite{ferre2021more,bansaye2019non,CV2020,villemonais2022quasi}.  

Theorem \ref{berry-esseen-1} will be proved at the third section. To prove it, we first need to show a central limit theorem for the $Q$-process satisfying \eqref{ergodicity-q-process}. In particular, up to my knowledge, the papers dealing with central limit theorems for Markov processes require stronger hypotheses than \eqref{ergodicity-q-process} (see the references provided in Section 2). That is why the second section aims to prove a central limit theorem for general ergodic Markov processes, which could be interesting and useful beyond the framework of quasi-stationarity. 

To conclude, the paper ends with an appendix showing the implication Assumption 1 $\Rightarrow$ Assumption 2 and Lemma \ref{lemma-q-proc} stated above. In particular, even if the existence of a quasi-ergodic distribution is quite classical assuming that Lemma \ref{lemma-q-proc} holds true (see for example \cite{CV2016} for a simple proof of this statement), the lemma itself is not clearly stated in the literature for processes satisfying Assumption \ref{assumption} (\cite{CV2016}, for example, states it under stronger conditions). This is why a short proof is provided in this appendix.          

\section{Central limit theorem for Markov processes}
\label{tcl}

This section aims to establish a central limit theorem for Markov processes satisfying the condition \eqref{ergodicity-q-process}. In the literature, central limit theorems for continuous-time Markov processes have, among others, been established in \cite{komorowski2012central,cattiaux2012central,lezaud2001chernoff}. In particular, the papers \cite{komorowski2012central,cattiaux2012central} made use of central limit theorems for martingales; the paper \cite{lezaud2001chernoff} used Kato's theory applied to analytically perturbed operators. 

In this paper, a central limit theorem will be proved for Markov processes studying the convergence of the moments of $\frac{1}{\sqrt{t}} \int_0^t f(X_s)ds$, for bounded functions $f$ such that $\beta(f) = 0$ and $\sigma_f^2 > 0$. Up to my knowledge, this method to establish a central limit theorem for (non-stationary) Markov processes is new. However, this method is difficult to apply for discrete-time processes; we refer to \cite{dobrushin1956central,derriennic2003central,kurtz1981central,haggstrom2005central,cuny2009pointwise} for central limit theorems for discrete-time Markov chains. 

In all this section, we deal with a general Markov process $(X_t)_{t \geq 0}$ defined on a state space $E'$. We denote by $(\Q_x)_{x \in E'}$ a family of probability measure such that, for all $x \in E'$, $\Q_x(X_0 = x) = 1$, for all probability measure $\mu \in \cM_1(E')$, $\Q_\mu := \int_{E'} \Q_x \mu(dx)$. We denote by $\E_\cdot^\Q$ and $\Cov_\cdot^\Q$ the expectation and the covariance associated to the probability measure $\Q_\cdot$, respectively. 

We emphasize that this section can be read independently on the rest of the paper. In particular, $(X_t)$ has no link with the $Q$-process with this section. 

In all what follows, we denote by $\cB_1(E')$ the set of the bounded by $1$ measurable functions defined over $E'$.

We introduce now the only assumption used all along this section: 
\begin{assumption}
\label{expo-ergodicity}
The process $(X_t)_{t \geq 0}$ admits an invariant measure $\beta$, and there exists a function $\psi : E' \to [c,+\infty)$ ($c > 0$) and two constants $C, \gamma > 0$ such that, for all $x \in E'$ and $t \geq 0$,
$$\|\Q_x(X_t \in \cdot) - \beta\|_{\psi} \leq C e^{-\gamma t} \psi(x).$$
\end{assumption}

In accordance with the introduction, we introduce, for all bounded function $f$, the variance
$$\sigma_f^2 := 2 \int_0^\infty \Cov_\beta^\Q(f(X_0),f(X_s))ds.$$

\subsection{Convergence of the moments of $\frac{1}{\sqrt{t}} \int_0^t f(X_s)ds$}

In this subsection, the following theorem will be proved. 

\begin{theorem}
\label{moment-ordre-pair} Assume that $(X_t)_{t \geq 0}$ satisfies Assumption \ref{expo-ergodicity}. Then
there exist a positive constants $C_1$ and a sequence of positive constant $(D_k)_{k \in \N}$ such that,
for all $k \in \Z_+$, $\mu \in \cM_1(E')$ such that $\mu(\psi) < + \infty$, $f \in \cB_1(E')$ such that $\beta(f) = 0$ and $t > 0$,
\begin{equation}\label{even-moment}\left|\E_\mu^\Q\left(\frac{1}{t^k} \left(\int_0^t f(X_s)ds\right)^{2k}\right) - \frac{(2k)!}{k!} \frac{\sigma_f^{2k}}{2^k}\right| \leq (2k)! D_k C_1 \frac{k}{(k-1)!} \frac{\mu(\psi)}{t}\end{equation}
and 
$$\lim_{t \to \infty} \E_\mu^\Q\left(\frac{1}{t^k \sqrt{t}} \left(\int_0^t f(X_s)ds\right)^{2k+1}\right) = 0.$$
In particular, for all $\mu \in \cM_1(E')$ such that $\mu(\psi) < + \infty$ and $f \in \cB_1(E')$ such that $\beta(f) = 0$ and $\sigma_f^2 > 0$,
$$\Q_\mu\left(\frac{1}{\sqrt{t}}\int_0^t f(X_s)ds \in \cdot\right) \underset{t \to \infty}{\overset{w}{\longrightarrow}} \cN(0,\sigma_f^2).$$
Moreover, a suitable sequence $(D_k)_{k \in \N}$ satisfying the inequalities \eqref{even-moment} is the one defined as:
$$D_k := \left(\left(\frac{C}{\gamma}(1+\frac{\beta(\psi)}{c})\right)^{k-1} \lor 1 \right) \times \left(\frac{C^2}{c} \lor \frac{C\beta(\psi)}{c^2\gamma} \right),~~~~\forall k \in \N.$$

\end{theorem}

Before proving Theorem \ref{moment-ordre-pair}, we need to prove two lemmata.

\begin{lemma} There exists a sequence of positive constants $(D_k)_{k \in \N}$ such that,
for all $f \in \cB_1(E')$ such that $\beta(f) = 0$ and $\sigma_f^2 > 0$, for all $k \in \N$, $\mu \in \cM_1(E')$ and $s_2 \leq \ldots \leq s_{2k}$, 
\begin{multline}\label{recurrence}\left|\E_\mu^\Q\left(\left[\int_0^{s_2}f(X_{s_1})ds_1\right] f(X_{s_2}) \ldots \left[\int_{s_{2k-2}}^{s_{2k}} f(X_{s_{2k-1}})ds_{2k-1}\right]f(X_{s_{2k}})\right) - \frac{\sigma_f^{2k}}{2^k}\right| \\ \leq D_k \mu(\psi) \sum_{i=0}^{k-1} (s_{2(i+1)}-s_{2i}+1)e^{-\gamma (s_{2(i+1)} - s_{2i})},\end{multline}
where $s_0 = 0$ by convention. \end{lemma}

\begin{proof}
We prove it by induction on $k$. We begin by showing the case $k = 1$. For all $\mu \in \cM_1(E')$ and $f \in \cB_1(E')$ and $t \geq 0$,
\begin{align}\E^\Q_{\mu}\left(\left[\int_0^tf(X_s)ds\right]f(X_t)\right) &= \int_0^t \E_{\mu}^\Q(f(X_s)f(X_t))ds \notag \\
&= \int_0^t \E_{\mu Q_{t-s}}^\Q(f(X_0)f(X_{s}))ds, \label{integarle}\end{align}
where we denote by $(Q_t)_{t \geq 0}$ the semigroup for the process $(X_t)_{t \geq 0}$.
By Assumption \ref{expo-ergodicity}, for all $x \in E'$ and $f \in \cB_1(E')$ such that $\beta(f) = 0$,
\begin{equation}\label{controle}|\E_x^\Q(f(X_0)f(X_s))| \leq  \E_x^\Q(|\E_{X_0}^\Q(f(X_s))|) \leq \frac{C}{c} \psi(x) e^{-\gamma s}.\end{equation}
Hence, by \eqref{integarle}, Assumption \ref{expo-ergodicity} and this last inequality, for all $t \geq 0$, $\mu \in \cM_1(E')$ such that $\mu(\psi)<+\infty$ and $f \in \cB_1(E')$ such that $\beta(f) = 0$, 
\begin{align}\left|\E^\Q_{\mu}\left(\left[\int_0^t f(X_s)ds\right]f(X_t)\right) - \int_0^t \E_{\beta}^\Q(f(X_0)f(X_{s}))ds\right| &\leq C \mu(\psi) \int_0^t e^{-\gamma (t-s)} \|\E_\cdot^\Q(f(X_0)f(X_s))\|_{\L^\infty(\psi)}ds \notag \\
&\leq C \mu(\psi) \int_0^t e^{-\gamma(t-s)} \frac{C}{c} e^{-\gamma s}ds \notag \\
&\leq \frac{C^2}{c} \mu(\psi) t e^{-\gamma t}.\label{control}\end{align}
Moreover, since $\beta(f) = 0$, by \eqref{controle}, for all $t \geq 0$,
\begin{equation}\left|\int_t^\infty \E_\beta^\Q(f(X_0)f(X_s))ds\right| \leq \int_t^\infty \frac{C}{c} \beta(\psi) e^{-\gamma s}ds \leq \frac{C \beta(\psi)}{c\gamma} e^{- \gamma t} \leq \frac{C \beta(\psi)}{c^2\gamma} \mu(\psi) e^{- \gamma t}. \label{control-2}\end{equation}
Hence, by definition of $\sigma_f^2$, there exists $D_1 > 0$ such that
$$\left|\E^\Q_{\mu}\left(\left[\int_0^t f(X_s)ds\right]f(X_t)\right) - \frac{\sigma_f^2}{2}\right| \leq D_1 \mu(\psi) (t+1)e^{-\gamma t}.$$
In particular, one can choose here $D_1 := \frac{C^2}{c} \lor \frac{C \beta(\psi)}{c^2\gamma}$. 
This concludes the base case. 

Let $k-1 \in \N$ be such that the hypothesis of induction is satisfied. Then, by the Markov property, 
\begin{multline}
\E_\mu^\Q\left(\int_0^{s_2}f(X_{s_1})ds_1 f(X_{s_2}) \ldots \int_{s_{2k-2}}^{s_{2k}} f(X_{s_{2k-1}})ds_{2k-1}f(X_{s_{2k}})\right)\\ = \E_\mu^\Q\left(\int_0^{s_2}f(X_{s_1})ds_1f(X_{s_2}) \E_{X_{s_2}}^\Q\left(\int_{s_2}^{s_4} f(X_{{s_3}-{s_2}})ds_3f(X_{{s_4}-{s_2}}) \ldots \int_{s_{2k-2}}^{s_{2k}} f(X_{s_{2k-1}-s_2})ds_{2k-1}f(X_{s_{2k}-s_2})\right)\right) \\
= \E_\mu^\Q\left(\int_0^{s_2}f(X_{s_1})ds_1f(X_{s_2}) \E_{X_{s_2}}^\Q\left(\int_{0}^{s_4-s_2} f(X_{{s_3}})ds_3f(X_{{s_4}-{s_2}}) \ldots \int_{s_{2k-2}-s_2}^{s_{2k}-s_2} f(X_{s_{2k-1}})ds_{2k-1}f(X_{s_{2k}-s_2})\right)\right). \label{grosse-equation}
\end{multline}
By hypothesis, for all $s_2 \leq s_4 \ldots \leq s_{2k}$,
\begin{multline}\left| \E_{X_{s_2}}^\Q\left(\int_{0}^{s_4-s_2} f(X_{{s_3}-{s_2}})ds_3f(X_{{s_4}-{s_2}}) \ldots \int_{s_{2k-2}-s_2}^{s_{2k}-s_2} f(X_{s_{2k-1}})ds_{2k-1}f(X_{s_{2k}-s_2})\right) - \frac{\sigma_f^{2k-2}}{2^{k-1}}\right| \\ \leq D_{k-1} \psi(X_{s_2}) \sum_{i=1}^{k-1} (s_{2(i+1)} - s_{2i} + 1)e^{-\gamma(s_{2(i+1)}-s_{2i})}.\label{bonne-inegalite}\end{multline}

Moreover, since $\beta(f) = 0$, for all $\mu \in \cM_1(E')$, for all $s_2 \geq 0$, for all $h \in \L^\infty(\psi)$,
\begin{equation}\label{controle-3}\left|\E_\mu^\Q\left(\int_0^{s_2} f(X_s)dsf(X_{s_2})h(X_{s_2})\right)\right| \leq \frac{C}{\gamma}(1+\frac{\beta(\psi)}{c}) \mu(\psi) \|h\|_{\L^\infty(\psi)},\end{equation}
where $C$ is the constant implied in Assumption \ref{expo-ergodicity}.
Indeed, for all $t \geq 0$, $\mu \in \cM_1(E')$ and $f,g \in \L^\infty(\psi)$,
$$\int_0^t \E_\mu^\Q\left[f(X_s)g(X_t)\right]ds = \int_0^t \E_\mu^\Q\left[f(X_s)\E_{X_s}(g(X_{t-s}))\right]ds.$$
Thus, by Assumption \ref{expo-ergodicity}, for all $t \geq 0$, $\mu \in \cM_1(E')$, $f \in \cB_1(E')$ and $g \in \L^\infty(\psi)$,
\begin{align*}\left|\int_0^t \E_\mu^\Q\left[f(X_s)g(X_t)\right]ds-\int_0^t \E_\mu^\Q\left[f(X_s)\right]\beta(g)ds\right| &\leq C \|g\|_{\L^\infty(\psi)} \int_0^t \mu(\psi)e^{-\gamma(t-s)}ds \\
&\leq \frac{C}{\gamma} \|g\|_{\L^\infty(\psi)}\mu(\psi).\end{align*}
Moreover, again by Assumption \ref{expo-ergodicity}, for all $t \geq 0$, $\mu \in \cM_1(E')$ and $f \in \cB_1(E')$ such that $\beta(f) = 0$,
$$\left|\int_0^t \E_\mu^\Q\left[f(X_s)\right]ds\right| \leq \frac{C}{c} \mu(\psi) \int_0^t e^{-\gamma s}ds \leq \frac{C}{c\gamma} \mu(\psi).$$
These two last inequalities applied to $g = f \times h$ imply \eqref{controle-3}. 

Now, denote 
$$h: x \mapsto \E_{x}^\Q\left(\int_{0}^{s_4-s_2} f(X_{{s_3}-{s_2}})ds_3f(X_{{s_4}-{s_2}}) \ldots \int_{s_{2k-2}-s_2}^{s_{2k}-s_2} f(X_{s_{2k-1}})ds_{2k-1}f(X_{s_{2k}-s_2})\right) - \frac{\sigma_f^{2k-2}}{2^{k-1}}.$$
Then, by \eqref{bonne-inegalite}, $h \in \L^\infty(\psi)$ and 
$$\|h\|_{\L^\infty(\psi)} \leq D_{k-1} \sum_{i=1}^{k-1} (s_{2(i+1)} - s_{2i} + 1)e^{-\gamma(s_{2(i+1)}-s_{2i})}.$$
Hence, by \eqref{grosse-equation}, \eqref{controle-3} and this last inequality, 
\begin{multline*}\left|\E_\mu^\Q\left(\int_0^{s_2}f(X_{s_1})ds_1 f(X_{s_2}) \ldots \int_{s_{2k-2}}^{s_{2k}} f(X_{2k-1})ds_{2k-1}f(X_{s_{2k}})\right) - \frac{\sigma_f^{2(k-1)}}{2^{k-1}} \E_\mu^\Q\left(\int_0^{s_2}f(X_{s_1})ds_1 f(X_{s_2})\right)\right| \\ \leq \frac{C}{\gamma}(1+\frac{\beta(\psi)}{c})\mu(\psi)\|h\|_{\L^\infty(\psi)} \leq \frac{C}{\gamma}(1+\frac{\beta(\psi)}{c})\mu(\psi)D_{k-1} \sum_{i=1}^{k-1} (s_{2(i+1)} - s_{2i} + 1)e^{-\gamma(s_{2(i+1)}-s_{2i})}. \end{multline*}

This and the case $k=1$ conclude the induction setting 
$$D_k := \left[D_{k-1} \times \left(\frac{C}{\gamma}(1+\frac{\beta(\psi)}{c})\right)\right] \lor D_1.$$
\end{proof}

We need also the following lemma.

\begin{lemma}
 For all $k \in \Z_+$, there exists $C_k \in (0,+\infty)$ such that,  for $t \geq 1$,
\begin{equation}\label{appendix}\int_{0 \leq s_2 \ldots \leq s_{2k} \leq t}  \sum_{i=0}^{k-1} (s_{2(i+1)} - s_{2i} + 1)e^{-\gamma(s_{2(i+1)}-s_{2i})} ds_2 \ldots ds_{2k} \leq C_k t^{k-1}.\end{equation} \end{lemma}

\begin{proof}
We prove \eqref{appendix} by induction on $k$. The case $k=1$ can easily be obtained by the reader for a given constant $C_1 > 0$. Now, assume that \eqref{appendix} holds true for $k-1 \in \N$. For all $t \geq 0$,
\begin{multline}
\sum_{i=0}^{k-1} \int_{0 \leq s_2 \leq \ldots \leq s_{2k} \leq t} (s_{2(i+1)} - s_{2i} + 1)e^{-\gamma (s_{2(i+1)}-s_{2i})}ds_2 \ldots ds_{2k}  \\
= \int_{0 \leq s_2 \leq \ldots \leq s_{2k} \leq t} \sum_{i=0}^{k-2} (s_{2(i+1)} - s_{2i}+1)e^{-\gamma (s_{2(i+1)} - s_{2i})}ds_2 \ldots ds_{2k}  \\
+ \int_{0 \leq s_2 \leq \ldots \leq s_{2k} \leq t} (s_{2k}-s_{2(k-1)}+1)e^{-\gamma (s_{2k}-s_{2(k-1)})}ds_2 \ldots ds_{2k} \\
= \int_0^t \left[\int_{0 \leq s_2 \leq \ldots \leq s_{2(k-1)} \leq s_{2k}} \sum_{i=0}^{k-2} (s_{2(i+1)}-s_{2i}+1)e^{-\gamma(s_{2(i+1)}-s_{2i})}ds_2 \ldots ds_{2(i-1)}\right]ds_{2k} \\
+ \int_{0 \leq s_2 \leq \ldots \leq s_{2k} \leq t} (s_{2k}-s_{2(k-1)}+1)e^{-\gamma (s_{2k} - s_{2(k-1)})}ds_2 \ldots ds_{2k}. \label{deuxieme-lemma}
\end{multline} 
By hypothesis, for all $t \geq 0$,
\begin{multline*}\int_0^t \left[\int_{0 \leq s_2 \leq \ldots \leq s_{2(k-1)} \leq s_{2k}} \sum_{i=0}^{k-2} (s_{2(i+1)}-s_{2i})e^{-\gamma(s_{2(i+1)}-s_{2i})}ds_2 \ldots ds_{2(i-1)}\right]ds_{2k} \\ \leq \int_0^t C_{k-1} s_{2k}^{k-2}ds_{2k} = C_{k-1} \frac{t^{k-1}}{k-1}.\end{multline*}
For all $t \geq 0$, the second term of \eqref{deuxieme-lemma} is equal to
\begin{align*}\int_{0 \leq s_{2(k-1)} \leq s_{2k} \leq t} (s_{2k}-s_{2(k-1)}+1)&e^{-\gamma (s_{2k}-s_{2(k-1)})} \left[\int_{0 \leq s_2 \leq \ldots \leq s_{2k-1}} ds_2 \ldots ds_{2(k-2)}\right]ds_{2(k-1)}ds_{2k} \\&= \int_{0 \leq r \leq s \leq t} (s-r+1)e^{-\gamma(s-r)}\frac{r^{k-2}}{(k-2)!}drds \\
&= \int_0^t \left[\int_r^t (s-r+1)e^{-\gamma(s-r)}ds\right]\frac{r^{k-2}}{(k-2)!}dr \\
&=\int_0^t\left( \int_0^{t-r} (u+1)e^{-\gamma u} du\right)\frac{r^{k-2}}{(k-2)!}dr \leq \frac{C_1}{(k-1)!} t^{k-1},\end{align*}
where $C_1 < + \infty$ is exactly the same constant as for the case $k=1$.  
Hence, \eqref{appendix} is proved with $C_k$ satisfying the relation $C_k = \frac{C_{k-1}}{k-1} + \frac{C_1}{(k-1)!}$. By induction, for all $k \geq 2$,
$$C_k = \frac{C_1}{(k-1)!} + \frac{C_1}{(k-2)!}.$$
\end{proof}

We can now prove Theorem \ref{moment-ordre-pair}.

\begin{proof}[Proof of Theorem \ref{moment-ordre-pair}]
We begin with the convergence of the even moment. For all $\mu \in \cM_1(E')$, $t \geq 0$, $f \in \cB_1(E')$ and $k \in \Z_+$,
\begin{multline}\label{even}\E_\mu^\Q\left(\left(\int_0^t f(X_s)ds\right)^{2k}\right) \\ = (2k)! \int_{0 \leq s_2 \leq \ldots \leq s_{2k} \leq t} \int_0^{s_2} \int_{s_2}^{s_4} \cdots \int_{s_{2k-2}}^{s_{2k}} \E_\mu^\Q(f(X_{s_1})f(X_{s_2}) \ldots f(X_{s_{2k-1}})f(X_{s_{2k}}))ds_1 \ldots ds_{2k}.\end{multline}
Then, assuming moreover that $\beta(f) = 0$, by \eqref{even},\eqref{recurrence} and \eqref{appendix},
\begin{equation}\label{control}\left|\E_\mu^\Q\left(\left(\int_0^t f(X_s)ds\right)^{2k}\right) - \frac{(2k)!}{k!} t^k \frac{\sigma_f^{2k}}{2^k}\right| \leq (2k)! D_k \mu(\psi) \times C_k t^{k-1},\end{equation}
which implies \eqref{even-moment}. Now, for all $\mu \in \cM_1(E')$, $t \geq 0$, $k \in \Z_+$ and $f \in \cB_1(E')$ such that $\beta(f) = 0$, 

\begin{align*}\E_\mu^\Q&\left(\left(\int_0^t f(X_s)ds\right)^{2k+1}\right) \\ &= (2k+1)! \int_0^t \E_\mu^\Q\left(f(X_s)\E_{X_{s}}^\Q\left[\int_{0 \leq s_2 \leq \ldots \leq s_{2k+1} \leq t-s} f(X_{s_2}) \ldots f(X_{s_{2k+1}})ds_2\ldots ds_{2k+1}\right]\right)ds \\
&= (2k+1) \int_0^t \E_\mu^\Q\left(f(X_s) \E_{X_s}^\Q\left(\left(\int_0^{t-s} f(X_u)du\right)^{2k}\right)\right)ds \\
&= (2k+1) \int_0^t \E_\mu^\Q\left(f(X_{t-s}) \E_{X_{t-s}}^\Q\left(\left(\int_0^{s} f(X_u)du\right)^{2k}\right)\right)ds .
\end{align*}
By \eqref{control} and using that $\E_\mu^\Q[\psi(X_{t-s})] \leq (\frac{\beta(\psi)}{c} + C) \mu(\psi)$ for all $\mu \in \cM_1(E')$ and $s \leq t$ (this is a consequence of Assumption \eqref{expo-ergodicity}), there exists $\hat{C} > 0$ such that
\begin{equation}\label{control-2}\left|\E_\mu^\Q\left(\left(\int_0^t f(X_s)ds\right)^{2k+1}\right) - \frac{(2k+1)!}{k!} \frac{\sigma_f^{2k}}{2^k} \int_0^t s^k \E_\mu^\Q\left(f(X_{t-s})\right)ds\right| \leq D_k \hat{C} \frac{k(2k+1)}{(k-1)!} \mu(\psi) \frac{t^k}{k}.\end{equation}
Since $\beta(f) = 0$, by Assumption \ref{expo-ergodicity}, for all $\mu \in \cM_1(E')$ and $s \leq t$,
\begin{equation}\label{encore-ergodicity}|\E_\mu^\Q[f(X_{t-s})]| \leq C \mu(\psi) e^{-\gamma (t-s)}.\end{equation}
For all $t > 0$ and $k \in \Z_+$,
\begin{equation}\label{inte}\frac{1}{t^{k+\frac{1}{2}}} \int_0^t s^k e^{-\gamma (t-s)}ds = \frac{e^{-\gamma t}}{t^{k + 1/2}} \int_0^t s^k e^{\gamma s} ds \leq \frac{e^{-\gamma t}}{\sqrt{t}} \int_0^t e^{\gamma s}ds \leq \frac{1}{\gamma \sqrt{t}}.\end{equation}
We deduce from \eqref{control-2}, \eqref{encore-ergodicity} and \eqref{inte} that there exists $\hat{C} > 0$ (different from the previous one) such that, for all $\mu \in \cM_1(E')$ such that $\mu(\psi) < + \infty$ and $f \in \cB_1(E')$ such that $\beta(f) = 0$, \begin{equation}\label{control-odd-moment}\left|\frac{1}{t^{k+1/2}} \E_\mu^\Q\left(\left(\int_0^t f(X_s)ds\right)^{2k+1}\right)\right| \leq D_k \times \hat{C} \left[\frac{(2k+1)!}{2^k k!} + \frac{2k+1}{(k-1)!}\right]\frac{\mu(\psi)}{\sqrt{t}}.\end{equation}
The central limit theorem is deduced from the method of moments. Now, concerning a suitable candidate for the sequence $(D_k)_{k \in \N}$, as proven in Lemma 1, a suitable candidate is the sequence defined recursively by
$$D_1 := \frac{C^2}{c} \lor \frac{C \beta(\psi)}{c^2\gamma},$$
$$D_k := \left[D_{k-1} \times \left(\frac{C}{\gamma} (1+\frac{\beta(\psi)}{c})\right)\right] \lor D_1,~~~~\forall k \geq 2,$$
in other words the sequence defined in Theorem \ref{moment-ordre-pair}. 
 \end{proof}

\subsection{A quantitative uniform CLT}
\label{useful-lemma}

The aim of this subsection is to prove the following result, which can be seen as an improved central limit theorem for $(X_t)_{t \geq 0}$.  

\begin{theorem}
\label{convergence-uniforme}
Assume that $(X_t)_{t \geq 0}$ satisfies Assumption \ref{expo-ergodicity}. Then,  for all $\mu \in \cM_1(E')$ such that $\mu(\psi) < \infty$, $f \in \cB_1(E')$ such that $\beta(f) = 0$ and $\sigma_f^2 > 0$, and $\omega \in \R$, 
\begin{equation}\label{uniform-convergence}\lim_{t \to \infty} \sup_{g \in \L^\infty(\psi) : \|g\|_{\L^\infty(\psi)} \leq 1} \left|\E^\Q_\mu\left[e^{\frac{i\omega}{\sqrt{t}}\int_0^t f(X_s)ds}g(X_t)\right] - \beta(g) e^{-\frac{\sigma_f^2 \omega^2}{2}}\right| = 0.\end{equation}
Moreover, for all $t > 0$ and $\mu \in \cM_1(E')$, one has 
$$ \sup_{g \in \L^\infty(\psi) : \|g\|_{\L^\infty(\psi)} \leq 1} \left|\E^\Q_\mu\left[e^{\frac{i\omega}{\sqrt{t}}\int_0^t f(X_s)ds}g(X_t)\right] - \beta(g) \E_\mu^\Q(e^{\frac{i \omega}{\sqrt{t}}\int_0^t f(X_s)ds})\right| \leq C \mu(\psi) e^{-\gamma t} + \frac{C |\omega|}{\sqrt{t}} \frac{\beta(\psi)+C \mu(\psi)}{\gamma}.$$
\end{theorem}

\begin{proof}[Proof of Theorem \ref{convergence-uniforme}.] 

For all $\mu \in \cM_1(E')$, $f \in \cB(E')$ such that $\beta(f) = 0$, $t \geq 0$, $k \in \Z_+$ and $g \in \L^\infty(\psi)$,
\begin{align*}
\E_\mu^\Q\left(\left(\int_0^t f(X_s)ds\right)^kg(X_t)\right) &= k \int_0^t \E_\mu^\Q\left(\left[\int_0^s f(X_u)du\right]^{k-1}f(X_s)g(X_t)\right)ds \\
&= k \int_0^t \E_\mu^\Q\left(\left[\int_0^s f(X_u)du\right]^{k-1}f(X_s)\E^\Q_{X_s}(g(X_{t-s}))\right)ds.
\end{align*}
Hence, for all $\mu \in \cM_1(E')$, $t \geq 0$, $f \in \cB_1(E')$ such that $\beta(f) = 0$, $k \in \Z_+$ and $g \in \L^\infty(\psi)$,
\begin{multline*}\E_\mu^\Q\left(\left(\int_0^t f(X_s)ds\right)^kg(X_t)\right) - \beta(g) \E_\mu^\Q\left(\left(\int_0^t f(X_s)ds\right)^k\right) \\= k \int_0^t \E_\mu^\Q\left(\left(\int_0^s f(X_u)du\right)^{k-1}f(X_s) [\E_{X_s}^\Q(g(X_{t-s}))-\beta(g)]\right)ds.\end{multline*}
Thus, using that $e^{\frac{i \omega}{\sqrt{t}} \int_0^t f(X_s)ds} = \sum_{k=0}^\infty \frac{i^k \omega^k}{t^{k/2}k!} (\int_0^t f(X_s)ds)^k$ for all $t \geq 0$, $\omega \in \R$, and $f \in \cB_1(E')$ such that $\beta(f) = 0$, then, using the above equality, for all $\mu \in \cM_1(E')$ and $g \in \L^\infty(\psi)$,
\begin{multline*}
\E_\mu^\Q\left(e^{\frac{i\omega}{\sqrt{t}} \int_0^t f(X_s)ds}g(X_t)\right) - \beta(g) \E_\mu^\Q\left(e^{\frac{i\omega}{\sqrt{t}} \int_0^t f(X_s)ds}\right) \\ = \sum_{k=0}^\infty \left(\frac{i \omega}{\sqrt{t}}\right)^k \frac{1}{k!} \left\{\E_\mu^\Q\left(\left(\int_0^t f(X_s)ds\right)^kg(X_t)\right) - \beta(g) \E_\mu^\Q\left(\left(\int_0^t f(X_s)ds\right)^k\right)\right\}\\
= \E_\mu^\Q(g(X_t)) - \beta(g) + \sum_{k=1}^\infty \left(\frac{i \omega}{\sqrt{t}}\right)^k \frac{1}{k!} \left\{\E_\mu^\Q\left(\left(\int_0^t f(X_s)ds\right)^kg(X_t)\right) - \beta(g) \E_\mu^\Q\left(\left(\int_0^t f(X_s)ds\right)^k\right)\right\}\\
= \E_\mu^\Q(g(X_t)) - \beta(g) + \frac{i \omega}{\sqrt{t}} \int_0^t \E_\mu^\Q\left(e^{\frac{i\omega}{\sqrt{s}} \int_0^s f(X_u)du}f(X_s) [\E_{X_s}^\Q(g(X_{t-s}))-\beta(g)]\right)ds.\end{multline*}
By Assumption \ref{expo-ergodicity} one has, for all $\mu \in \cM_1(E')$, $\Q_\mu$-almost surely and for all $s \leq t$ and $g \in \L^\infty(\psi)$,
$$|\E_{X_s}^\Q(g(X_{t-s})) - \beta(g)| \leq C \|g\|_{\L^\infty(\psi)} \psi(X_s) e^{-\gamma (t-s)}.$$
Thus, for all $\mu \in \cM_1(E')$, $t > 0$, $\omega \in \R$, $g \in \L^\infty(\psi)$ and $f \in \cB_1(E')$ such that $\beta(f) = 0$, 
\begin{multline*}\left|\E_\mu^\Q\left(e^{\frac{i\omega}{\sqrt{t}} \int_0^t f(X_s)ds}g(X_t)\right) - \beta(g) \E_\mu^\Q\left(e^{\frac{i\omega}{\sqrt{t}} \int_0^t f(X_s)ds}\right)\right|\\ \leq C \mu(\psi) \|g\|_{\L^\infty(\psi)}e^{-\gamma t} + \frac{C \|g\|_{\L^\infty(\psi)} |\omega|}{\sqrt{t}} \int_0^t  e^{-\gamma (t-s)} \E_\mu^\Q(\psi(X_s))ds.\end{multline*}
Since $|\E_\mu^\Q(\psi(X_s)) - \beta(\psi)| \leq C \mu(\psi) e^{-\gamma s}$ for all $s \geq 0$, one has that 
$$\sup_{t \geq 0}  \int_0^t  e^{-\gamma (t-s)} \E_\mu^\Q(\psi(X_s))ds \leq  \sup_{t \geq 0}  \int_0^t  e^{-\gamma (t-s)} [\beta(\psi) +  C \mu(\psi) e^{-\gamma s}]ds \leq \frac{\beta(\psi) + C \mu(\psi)}{\gamma}.$$
These two last inequalities and Theorem \ref{moment-ordre-pair} imply \eqref{uniform-convergence} and conclude the proof.
\end{proof}

\section{Proof of Theorem \ref{berry-esseen-1}}
In this section, $(X_t)_{t \geq 0}$ refers again to the process living in $E \cup \{\d\}$ and absorbed at $\d$.

We now prove Theorem \ref{berry-esseen-1}, divided in three steps. \medskip
 
\underline{\textit{Step 1.}} For all $\mu \in \cM_1(E)$, $t \geq 0$ and $f \in \cB_1(E)$, $g \in \L^\infty(\psi_1)$ and $k \in \Z_+$,
\begin{multline*}\E_\mu\left(\left[\int_0^t f(X_s)ds\right]^k g(X_t) \middle| \tau_\d > t\right) = k! \int_{0 \leq s_1 \leq \ldots \leq s_k \leq t} \E_\mu(f(X_{s_1}) f(X_{s_2}) \ldots f(X_{s_k}) g(X_t) | \tau_\d > t) ds_1 \ldots ds_k \\
= k! \int_0^t \E_\mu\left(\left[\int_{0 \leq s_1 \leq \ldots \leq s_{k-1} \leq s} f(X_{s_1}) \ldots f(X_{s_{k-1}}) ds_1 \ldots ds_{k-1}\right]f(X_s)g(X_t)\middle| \tau_\d > t\right)ds \\
= k \int_0^t \E_\mu\left(\left[\int_0^s f(X_u)du\right]^{k-1}f(X_s)g(X_t)\middle| \tau_\d > t\right)ds \\
= k \int_0^t \frac{1}{\P_\mu(\tau_\d > t)}\E_\mu\left(\left[\int_0^s f(X_u)du\right]^{k-1}f(X_s) \E_{X_s}(g(X_{t-s})\1_{\tau_\d > t-s}) \1_{\tau_\d > s}\right)ds.\end{multline*}
For all $s \leq t$, $\mu \in \cM_1(E)$, $g \in \L^\infty(\psi_1)$ and $x \in E$, denote 
$$C_{\mu,g}(s,t,x) := \frac{\mu(\eta)}{e^{\lambda_0 s}} e^{\gamma(t-s)} \left\{ \frac{\E_x(g(X_{t-s})\1_{\tau_\d > t-s})}{\P_\mu(\tau_\d > t)} - \frac{e^{\lambda_0 s}\eta(x)\alpha(g)}{\mu(\eta)}\right\}.$$
Thus, for all $\mu \in \cM_1(E)$, $f \in \cB_1(E)$ such that $\beta(f) = 0$, $g \in \L^\infty(\psi_1)$, $k \in \Z_+$ and $t \geq 0$,
\begin{multline*}
    \E_\mu\left(\left[\int_0^t f(X_s)ds\right]^k g(X_t) \middle| \tau_\d > t\right) - \alpha(g) \E^\Q_{\eta \circ \mu}\left(\left[\int_0^t f(X_s)ds\right]^k \right) \\
    = k\times e^{-\gamma t} \int_0^t e^{\gamma s} \E^\Q_{\eta \circ \mu}\left(\left(\int_0^s f(X_u)du\right)^{k-1} \frac{f(X_s)C_{\mu,g}(s,t,X_s)}{\eta(X_s)}\right)ds.
\end{multline*}
Similarly to the proof of Theorem \ref{convergence-uniforme}, for all $t > 0$, $\omega \in \R$, $\mu \in \cM_1(E)$ and $f$ such that $\beta(f) = 0$,
\begin{multline}
    \E_\mu\left(e^{\frac{i \omega}{\sqrt{t}}\int_0^t f(X_s)ds} \middle| \tau_\d > t\right) - \E_{\eta \circ \mu}^\Q\left[e^{\frac{i \omega}{\sqrt{t}}\int_0^t f(X_s)ds}\right] \\
    = \sum_{k=1}^\infty \left(\frac{i \omega}{\sqrt{t}}\right)^k \frac{1}{k!} \left\{\E_\mu\left(\left[\int_0^t f(X_s)ds\right]^k \middle| \tau_\d > t\right) - \E^\Q_{\eta \circ \mu}\left(\left[\int_0^t f(X_s)ds\right]^k \right)\right\} 
    \\ = \sum_{k=1}^\infty \frac{i^k \omega^k}{t^{k/2}} \frac{1}{(k-1)!}  e^{-\gamma t} \int_0^t e^{\gamma s} \E^\Q_{\eta \circ \mu}\left(\left(\int_0^s f(X_u)du\right)^{k-1} \frac{f(X_s)C_{\mu,\1_E}(s,t,X_s)}{\eta(X_s)}\right)ds  \\
    =  e^{-\gamma t} \int_0^t e^{\gamma s} \frac{i \omega}{\sqrt{t}} \E^\Q_{\eta \circ \mu}\left(e^{\frac{ i \omega}{\sqrt{s}}\int_0^s f(X_u)du} \frac{f(X_s)C_{\mu,\1_E}(s,t,X_s)}{\eta(X_s)}\right)ds. \label{moment}
\end{multline}
\underline{\textit{Step 2.}} By triangular inequality, for all $s \leq t$ and $x \in E$,
\begin{multline}|C_{\mu,g}(s,t,x)| \leq \frac{\mu(\eta)}{e^{\lambda_0 s}} e^{\gamma(t-s)} \left\{\left|\frac{\E_x[g(X_{t-s})\1_{\tau_\d > t-s}]}{\P_\mu(\tau_\d > t)} - \frac{e^{-\lambda_0 (t-s)}\eta(x) \alpha(g)}{\P_\mu(\tau_\d > t)}\right|\right. \\ + \left.\left|\frac{e^{-\lambda_0(t-s)}\eta(x)\alpha(g)}{\P_\mu(\tau_\d > t)} - \frac{e^{\lambda_0 s} \eta(x) \alpha(g)}{\mu(\eta)}\right|\right\}.\label{triangular}\end{multline}
By \eqref{ergodicity},
$$ \frac{\mu(\eta)}{e^{\lambda_0 s}} e^{\gamma(t-s)} \left|\frac{\E_x[g(X_{t-s})\1_{\tau_\d > t-s}]}{\P_\mu(\tau_\d > t)} - \frac{e^{-\lambda_0 (t-s)}\eta(x) \alpha(g)}{\P_\mu(\tau_\d > t)}\right| \leq C \|g\|_{\L^\infty(\psi_1)} \psi_1(x) \mu(\eta) \frac{e^{-\lambda_0 t}}{\P_\mu(\tau_\d > t)}.$$
Again by \eqref{ergodicity},
$$e^{\lambda_0 t}\P_\mu(\tau_\d > t) \geq \mu(\eta) - C \mu(\psi_1)e^{-\gamma t}.$$
Hence, for all $t \geq \frac{1}{\gamma} \log\left(\frac{2C \mu(\psi_1)}{\mu(\eta)}\right)$, 
\begin{align*}
\frac{\mu(\eta)}{e^{\lambda_0 t}\P_\mu(\tau_\d > t)} &\leq \frac{1}{1 - C \frac{\mu(\psi_1)}{\mu(\eta)}e^{-\gamma t}} \\
&\leq 1 + 2C \frac{\mu(\psi_1)}{\mu(\eta)}e^{-\gamma t} \leq 2.\end{align*}
For the second part of the right-hand side of the inequality \eqref{triangular},
\begin{align*}
\frac{\mu(\eta)}{e^{\lambda_0 s}} e^{\gamma(t-s)} \left|\frac{e^{-\lambda_0(t-s)}\eta(x)\alpha(g)}{\P_\mu(\tau_\d > t)} - \frac{e^{\lambda_0 s} \eta(x) \alpha(g)}{\mu(\eta)}\right| &= \eta(x) |\alpha(g)|e^{\gamma(t-s)} \left|\frac{\mu(\eta)}{e^{\lambda_0 t} \P_\mu(\tau_\d > t)} - 1 \right| \\
&\leq C \eta(x) |\alpha(g)| e^{-\gamma s} \frac{C \mu(\psi_1)}{e^{\lambda_0 t} \P_\mu(\tau_\d > t)} \\
&\leq C \eta(x) |\alpha(g)| 2 C \frac{\mu(\psi_1)}{\mu(\eta)}.\end{align*}
Hence, these inequalities, the fact that $|\alpha(g)| \leq \|g \|_{\L^\infty(\psi_1)} \alpha(\psi_1)$ and \eqref{triangular} imply the existence of a constant $C' > 0$ such that, for all $s \leq t$ such that $t \geq \frac{1}{\gamma} \log\left(\frac{2C \mu(\psi_1)}{\mu(\eta)}\right)$ and $x \in E$, 
\begin{equation}\label{majoration}\left|C_{\mu,g}(s,t,x)\right| \leq C'\|g\|_{\L^\infty(\psi_1)} \left[\psi_1(x) + \frac{\mu(\psi_1)}{\mu(\eta)} \eta(x)\right].\end{equation}

\underline{\textit{Last step.}} By using this last inequality \eqref{majoration} in \eqref{moment}, one obtains that, for all $f \in \cB(E)$ such that $\beta(f) = 0$, $\omega \in \R$ and $t \geq \frac{1}{\gamma} \log\left(\frac{2C \mu(\psi_1)}{\mu(\eta)}\right)$, 
\begin{align}\left|\E_\mu\left(e^{\frac{i \omega}{\sqrt{t}}\int_0^t f(X_s)ds} \middle| \tau_\d > t\right) - \E_{\eta \circ \mu}^\Q\left[e^{\frac{i \omega}{\sqrt{t}}\int_0^t f(X_s)ds}\right]\right| &\leq  e^{-\gamma t} \int_0^t e^{\gamma s} \frac{|\omega|}{\sqrt{t}} \E^\Q_{\eta \circ \mu}\left( \frac{|C_{\mu,\1_E}(s,t,X_s)|}{\eta(X_s)}\right)ds \notag \\ &\leq C' \| \1_E\|_{\L^\infty(\psi_1)} e^{- \gamma t} \int_0^t e^{\gamma s} \frac{|\omega|}{\sqrt{t}} \left(\E_{\eta \circ \mu}^\Q(\psi(X_s)) + \frac{\mu(\psi_1)}{\mu(\eta)}\right)ds \notag \\ &\leq C'' \frac{|\omega|}{\sqrt{t}} \frac{\mu(\psi_1)}{\mu(\eta)}, \label{decay}\end{align}
where $C'' > 0$. This, combined with Theorem \ref{moment-ordre-pair}, proves Theorem \ref{berry-esseen-1}. 

\begin{remark} The presence of $1/\sqrt{t}$ in the last inequality suggests the idea that a Berry-Esseen inequality holds true for Markov processes conditioned not to be absorbed satisfying Assumption 1. In reality, this last upper-bound does not allow directly to deduce such a result.

As a matter of fact, an approach would be to consider the inequality (from \cite{fellerintroduction}),
\begin{multline}d_{\Kolm}\left(\P_\mu(\frac{1}{\sqrt{t}} \int_0^t f(X_s)ds \in \cdot | \tau_\d > t), \Q_{\eta \in \mu}(\frac{1}{\sqrt{t}} \int_0^t f(X_s)ds \in \cdot)\right) \\ \leq \int_{-W}^W \frac{\left|\E_\mu\left(e^{\frac{i \omega}{\sqrt{t}}\int_0^t f(X_s)ds} \middle| \tau_\d > t\right) - \E_{\eta \circ \mu}^\Q\left[e^{\frac{i \omega}{\sqrt{t}}\int_0^t f(X_s)ds}\right]\right|}{|\omega|} d\omega + \frac{\tilde{C}}{W}, \label{lezaud}\end{multline}
where $\tilde{C} > 0$, holding true for all $W> 0$. In particular, it is visible that the inequality \eqref{decay} and Theorem \ref{convergence-uniforme} are not enough to deduce a Berry-Esseen theorem. 

The inequality \eqref{lezaud} is in particular used in \cite{lezaud2001chernoff} to state a Berry-Esseen theorem for reversible Markov processes, using spectral arguments applied to perturbated operators. This approach could certainly be adapted to prove a similar result in the quasi-stationary framework. \end{remark}

\section*{Appendix: Proof of Lemma \ref{lemma-q-proc} and Assumption 1 $ \Rightarrow$ Assumption 2.}

This little section is devoted to the proof of Lemma \ref{lemma-q-proc}, needed to justify the existence of a quasi-ergodic distribution and the convergence \eqref{qed}. The implication Assumption 1 $\Rightarrow$ Assumption 2 will also be proved in this short proof. 

\begin{proof}[Proof of Lemma \ref{lemma-q-proc}]
Assume Assumption 1. Let $t \geq 0$ and $\Gamma \in \cF_t$. Then, for all $T \geq t$ and $x \in E$,
$$e^{\lambda_0 T} \P_x(\Gamma , \tau_\d > T) = e^{\lambda_0 T} \E_x(\1_{\Gamma, \tau_\d > t} \P_{X_t}(\tau_\d > T-t)) = e^{\lambda_0 t} \E_x(\1_{\Gamma, \tau_\d > t} e^{\lambda_0(T-t)}\P_{X_t}(\tau_\d > T-t)).$$By Assumption 1, since $\psi_1 \geq 1$, for all $T \geq t$ and $x \in E$,
\begin{equation} \label{convergence} |e^{\lambda_0 T} \P_x(\Gamma, \tau_\d > T) - e^{\lambda_0 t} \E_x(\1_{\Gamma, \tau_\d > t} \eta(X_t))| \leq C e^{\lambda_0 t} \E_x(\psi_1(X_t) \1_{\tau_\d > t}) e^{- \gamma (T-t)}.\end{equation}
Assumption 1 implies that \begin{equation} \label{majoration-lemma-q-proc} e^{\lambda_0 t} P_t \psi_1(x) \leq \alpha(\psi_1) + C \psi_1(x) e^{-\gamma t} < + \infty, \end{equation} so that the previous inequality entails that, for all $x \in E$, 
$$\lim_{T \to \infty} e^{\lambda_0 T} \P_x(\Gamma, \tau_\d > T) = e^{\lambda_0 t} \E_x(\1_{\Gamma, \tau_\d > t} \eta(X_t)).$$
Since the previous inequality holds true for $\Gamma = \1_E$ and $t = 0$, we deduce that, for all $x \in E'$,
$$\lim_{T \to \infty} \P_x(\Gamma | \tau_\d > T) = e^{\lambda_0 t} \frac{\E_x(\eta(X_t)\1_{\Gamma, \tau_\d > t})}{\eta(x)},$$
which proves the first point of Assumption 2 setting, for all $x \in E'$ and $\Gamma \in \cF_t$,
$$\Q_x(\Gamma) := \E_x(e^{\lambda_0 t} \eta(X_t)\1_{\Gamma, \tau_\d > t})/\eta(x),$$
proving therefore the equality \eqref{expr-q-proc}. 
By definition of $\eta$ and $\alpha$, it is easy to check that $\beta$ is an invariant measure for the $Q$-process.
Moreover, by Assumption 1, for all $f \in \L^\infty(\psi)$ and $x \in E'$, 
$$|\E_x^\Q(f(X_t)) - \beta(f)| = \frac{|e^{\lambda_0 t} P_t(f\eta)(x) - \eta(x) \alpha(f\eta)|}{\eta(x)} \leq C \psi(x) e^{-\gamma t} \|f\|_{\L^\infty(\psi)},$$
which confirms therefore the implication Assumption 1 $\Rightarrow$ Assumption 2. It remains therefore to show the exponential convergence of the function $T \mapsto \P_\mu(\Gamma | \tau_\d > T)$ to $\Q_{\eta \circ \mu}(\Gamma)$, for all $\mu \in \cM_1(E)$ such that $\mu(\eta) > 0$ and $\mu(\psi_1) < + \infty$. To do so, fix such a probability measure $\mu$. Integrating the inequality \eqref{convergence} over $\mu(dx)$, for all $T \geq \frac{1}{\gamma} \log(2C \mu(\psi_1)/\mu(\eta))$,
$$\frac{e^{\lambda_0 t}\E_\mu(\1_{\Gamma, \tau_\d > t}\eta(X_t)) - C e^{\lambda_0 t} \mu P_t \psi_1 e^{-\gamma (T -t)}}{\mu(\eta) + C \mu(\psi_1) e^{-\gamma T}} \leq \P_\mu(\Gamma | \tau_\d > T) \leq \frac{e^{\lambda_0 t}\E_\mu(\1_{\Gamma, \tau_\d > t}\eta(X_t)) + C e^{\lambda_0 t} \mu P_t \psi_1 e^{-\gamma (T -t)}}{\mu(\eta) - C \mu(\psi_1) e^{-\gamma T}}.$$ 
Since $1/(1-x) \leq 1+2x$ for all $x \in (0,1/2]$, for all $T \geq \log\left(\frac{2C \mu(\psi_1)}{\mu(\eta)}\right)$,
\begin{multline*}\frac{e^{\lambda_0 t}\E_\mu(\1_{\Gamma, \tau_\d > t}\eta(X_t)) + C e^{\lambda_0 t} \mu P_t \psi_1 e^{-\gamma (T -t)}}{\mu(\eta) - C \mu(\psi_1) e^{-\gamma T}} \\ \leq \left(\frac{e^{\lambda_0 t}\E_\mu(\1_{\Gamma, \tau_\d > t}\eta(X_t))}{\mu(\eta)} + C \frac{e^{\lambda_0 t}\mu P_t \psi_1 e^{-\gamma (T -t)}}{\mu(\eta)}\right)\left(1 + 2   C \frac{ \mu(\psi_1) e^{-\gamma T}}{\mu(\eta)}\right) \\ \leq \frac{e^{\lambda_0 t}\E_\mu(\1_{\Gamma, \tau_\d > t}\eta(X_t))}{\mu(\eta)} +  2C \frac{\mu(\psi_1)}{\mu(\eta)} e^{- \gamma T} + 2 C \frac{e^{\lambda_0 t}\mu P_t \psi_1 e^{-\gamma (T -t)}}{\mu(\eta)}  \\ \leq \frac{e^{\lambda_0 t}\E_\mu(\1_{\Gamma, \tau_\d > t}\eta(X_t))}{\mu(\eta)} + 2 C (\alpha(\psi_1) + C + 1) \frac{\mu(\psi_1) e^{-\gamma (T -t)}}{\mu(\eta)},\end{multline*}
where we used \eqref{majoration-lemma-q-proc}.
In the same vein, there exists a constant $C' > 0$ such that, for all $T \geq \frac{1}{\gamma} \log(2C \mu(\psi_1)/\mu(\eta))$, 
$$\frac{e^{\lambda_0 t}\E_\mu(\1_{\Gamma, \tau_\d > t}\eta(X_t)) - C e^{\lambda_0 t}\mu P_t \psi_1 e^{-\gamma (T -t)}}{\mu(\eta) + C e^{\lambda_0 t}\mu P_t \psi_1 e^{-\gamma(T-t)}} \geq \frac{e^{\lambda_0 t}\E_\mu(\1_{\Gamma, \tau_\d > t}\eta(X_t))}{\mu(\eta)} - C' \frac{\mu(\psi_1) e^{-\gamma (T -t)}}{\mu(\eta)}.$$
In conclusion, there exists a constant $C'> 0$ such that, for all $\mu \in \cM_1(E)$ such that $\mu(\psi_1) < + \infty$ and $\mu(\eta) > 0$, $t \geq 0$, $\Gamma \in \cF_t$ and $T \geq t \lor \frac{1}{\gamma} \log(2C\mu(\psi_1)/\mu(\eta))$,
$$|\P_\mu(\Gamma | \tau_\d > T) - \Q_{\eta \circ \mu}(\Gamma)| \leq C' \frac{\mu(\psi_1) e^{-\gamma (T -t)}}{\mu(\eta)}.$$
To generalize this inequality for all $T \geq t$, it is enough to remark that, if $t \leq T$ and $T \leq \frac{1}{\gamma} \log(2C \mu(\psi_1)/\mu(\eta))$, then for all $\Gamma \in \cF_t$,
$$|\P_\mu(\Gamma | \tau_\d > t) - \Q_{\eta \circ \mu}(\Gamma)| \leq 2 \leq 2 \times 2C \frac{\mu(\psi_1)}{\mu(\eta)} \times e^{-\gamma (T-t)},$$
which concludes the proof.  
\end{proof}

%%%%%%%%%%%%%%%%%%%%%%%%%%%%%%%%%%%%%%%%%%%%%%%%%%%%%%%%%%%%%%%%%%%
%%                                                               %%
%% Supplementary Material, if any, should be provided in         %%
%% {supplement} environment  with title and short description.   %%
%%                                                               %%
%%%%%%%%%%%%%%%%%%%%%%%%%%%%%%%%%%%%%%%%%%%%%%%%%%%%%%%%%%%%%%%%%%%

%%%%%%%%%%%%%%%%%%%%%%%%%%%%%%%%%%%%%%%%%%%%%%%%%%%%%%%%%%%%%%%%%%%
%%                                                               %%
%% Use the two commands below for producing your bibliography    %%
%% with bibtex, then comment again the commands and include the  %%
%% content of the .bbl file in this file below the commands.     %%
%%                                                               %%
%%%%%%%%%%%%%%%%%%%%%%%%%%%%%%%%%%%%%%%%%%%%%%%%%%%%%%%%%%%%%%%%%%%

\bibliography{biblio-william}
\bibliographystyle{plain}

% add below the content of your .bbl file produced by bibtex.

%%%%%%%%%%%%%%%%%%%%%%%%%%%%%%%%%%%%%%%%%%%%%%%%%%%%%%%%%%%%%%%%%%%
%%                                                               %%
%% You may add acknowledgments (optional).                       %%
%%                                                               %%
%%%%%%%%%%%%%%%%%%%%%%%%%%%%%%%%%%%%%%%%%%%%%%%%%%%%%%%%%%%%%%%%%%%

%%%%%%%%%%%%%%%%%%%%%%%%%%%%%%%%%%%%%%%%%%%%%%%%%%%%%%%%%%%%%%%%%%%
%%                                                               %%
%% You have reached the end of your document.                    %%
%%                                                               %%
%%%%%%%%%%%%%%%%%%%%%%%%%%%%%%%%%%%%%%%%%%%%%%%%%%%%%%%%%%%%%%%%%%%

\end{document}